\documentclass{article}

\usepackage[english]{babel}
\usepackage[utf8]{inputenc}%coding of writteninput
\usepackage[T1]{fontenc}%vectorized fonts (cm-super package)

\usepackage{amsfonts}
\usepackage{amsmath, amsthm, amssymb}
\usepackage{graphicx}

\usepackage{enumitem}

\usepackage{color}

\newtheorem{theorem}{Theorem}
\theoremstyle{plain}

\newtheorem{assumption}{Assumption}
\theoremstyle{plain}

\newtheorem{remark}[theorem]{Remark}
\theoremstyle{plain}

\newtheorem{definition}[theorem]{Definition}
\theoremstyle{plain}

\newtheorem{corollary}[theorem]{Corollary}
\theoremstyle{plain}

\newcommand{\dx}{\,dx}
\newcommand{\hp}{s}

\author{Christian Kahle}

\date{\today}

\title{A $L^\infty$ bound for the Cahn--Hilliard equation with relaxed
non-smooth free energy}

\begin{document}

\maketitle

\begin{abstract}
Phase field models are widely used to describe multiphase systems. Here a smooth
indicator function, called phase field, is used to describe the spatial
distribution of the phases under investigation. Material properties like
density or viscosity are introduced as given functions of the phase field. 
These parameters typically have physical bounds to fulfil, e.g. positivity of
the density. To guarantee these properties, uniform bounds on the phase field
are of interest. 

In this work we derive a uniform bound on the solution of the
Cahn--Hilliard system, where we use the double-obstacle free energy, that is
relaxed by Moreau--Yosida relaxation.
\end{abstract}

\section{Introduction}

Phase field models are a common approach to describe fluid systems of two or
more components and to deal with the complex topology changes that might appear
in such systems.
One of the basic models is the Cahn--Hilliard system
\cite{CahnHilliard} that models spinodal decomposition of a binary metal alloy
with components having essentially the same density. Particles are only
transported by diffusion.
Based on this model several extensions to transport by convection (model 'H'
\cite{HohenbergHalperin1977}) and additionally to fluids with different
densities are proposed throughout the literature, see
\cite{Lowengrub_CahnHilliard_and_Topology_transitions,
Boyer_two_phase_different_densities,
Ding_Spelt_Shu_diffuse_interface_model_diff_density,
AbelsGarckeGruen_CHNSmodell}.
In those models typically the density and viscosity of the two fluid components
are introduced as given functions of a phase field that is introduced to
describe their spatial distribution and that is the solution of a Cahn--Hilliard type
equation.
In general for the Cahn--Hilliard equation with free energies, which allow
non-physical values of the phase field, no $L^\infty(\Omega)$ bounds are
available.
As a consequence, in general situations one can not guarantee that the density
and the viscosity of the fluids stay positive, i.e. one might run into
non-physical data.

In this work we summarize, combine and extend results on the analytical
treatment of the Cahn--Hilliard equation with double-obstacle free energy, which
is relaxed using Moreau--Yosida relaxation.
The aim of this work is, to help later work on Cahn--Hilliard type models by
providing $L^\infty(\Omega)$ bounds on the violation of the physical meaningful
values of the phase field.

Using Moreau--Yosida relaxation for the treatment of the Cahn--Hilliard equation
with double-obstacle free energy is first analytically investigated in
\cite{HintermuellerHinzeTber}. Therein especially the convergence  of the
solutions of the relaxed system to the solution of the double-obstacle system is
shown. One main ingredient is the interpretation of the Cahn--Hilliard equation
as the first order optimality conditions of a suitable optimal control problem with
box constraints in $H^1(\Omega)$ on the control. 

On the other hand, in
\cite{HintermuellerSchielaWollner_length_PD_path_MY_path_following},
a typical optimal control problem with box constraints in $C(\bar{\Omega})$ on
the state is investigated. Here the constraints are treated using Moreau--Yosida
relaxation. The authors provide decay rates for the $L^{\infty}(\Omega)$ norm of
the violation of the box constraints in terms of the relaxation parameter. The
proof relies on higher regularity of the state, i.e. H\"older regularity is used.

The results in
\cite{HintermuellerSchielaWollner_length_PD_path_MY_path_following}
apply for the case of Dirichlet boundary data on the state, while in 
\cite{Hintermueller_Keil_Wegner__OPT_CHNS_density} these results are used for
the Cahn--Hilliard equation with Neumann boundary data,
to prove convergence of solutions for a relaxed  equation to the solutions of
a Cahn--Hilliard system with double-obstacle free energy. However, in
\cite{Hintermueller_Keil_Wegner__OPT_CHNS_density} no convergence rate is
provided.

Here we combine the aforementioned results to obtain a $L^\infty(\Omega)$ bound
on the violation of the physically meaningful values of the solution of the
Cahn--Hilliard equation that can later be used 
in more sophisticated models where bounds on parameters, like the density, that depend
on the solution of the Cahn--Hilliard equation, are required.

The paper is organized as follows.
In Section \ref{sec:CH} we introduce the time discrete Cahn--Hilliard system
with double-obstacle free energy and its relaxation using Moreau--Yosida
relaxation. We further summarize results from
\cite{HintermuellerHinzeTber}.
In Section \ref{sec:Linfty1} we apply proofs from
\cite{HintermuellerSchielaWollner_length_PD_path_MY_path_following} to obtain a
$L^1(\Omega)$ bound and based on this a first $L^\infty(\Omega)$ bound on the
constraint violation. Using a structural assumption, in Section
\ref{sec:Linfty2} we improve this bound.
Finally a numerical validation is carried out in Section \ref{sec:numerics}.

\section{The Cahn--Hilliard system with double-obstacle free energy}
\label{sec:CH}
In the following, let $\Omega \subset \mathbb R^n$, $n\in \{2,3\}$ denote an
 open bounded domain, that fulfils the cone condition, see
\cite{Adams_SobolevSpaces}. Its outer normal we denote by $\nu_\Omega$.

We further use usual notation for Sobolev spaces defined on $\Omega$, see
\cite{Adams_SobolevSpaces}, i.e. $W^{k,p}(\Omega)$ denotes the space of
functions that admit weak derivatives up to order $k$ that are
Lebesgue-integrable to the power $p$. For $p=2$ we write
$H^k(\Omega):=W^{k,2}(\Omega)$ and for $k=0$ we write $L^p(\Omega) := W^{0,p}(\Omega)$.
For $v\in W^{k,p}(\Omega)$ its norm is denoted by $\|v\|_{W^{k,p}(\Omega)}$.

Moreover we introduce
\begin{align*}
\mathcal{K} = 
\{ v \in H^1(\Omega)\,|\, |v|\leq 1 a.e.\},\\
V_0 = \{ v \in H^1(\Omega)\,|\, (v,1) = 0 \}.
\end{align*}

For a fixed time $t$ we consider an alloy consisting of two components $A$ and
$B$ and we are interested in the evolution of this alloy over time.
 For the description of the distribution of the two components we introduce a
phase field $\varphi$ that serves as a binary indicator
function in the sense, that $\varphi(x) = 1$ indicates pure fluid of component
$A$ at point $x\in \Omega$, while $\varphi(x) = -1$ indicates pure fluid of
component $B$.
We further assume, that the transition zone $\Gamma_\epsilon$ between $A$ and
$B$ is of positive thickness, proportional to $\epsilon$, and that both
components are mixed therein. 
The phase field admits values from the interval $(-1,1)$ in this region.

We introduce  the Ginzburg--Landau energy  of the system by
\begin{align*}
  GL(\varphi) = \int_\Omega\frac{\epsilon}{2}|\nabla \varphi|^2 +
  \frac{1}{\epsilon} W(\varphi)\dx.
\end{align*}
Here $W(\varphi)$ is the so called free energy density and is of double well
type, i.e.
it admits exactly two minima at $\pm1$ and is positive everywhere else.
We denote the first variation of $GL$ by
\begin{align*}
  \mu := -\epsilon\Delta \varphi + \epsilon^{-1}W^\prime(\varphi) 
\end{align*}
and introduce a mass conserving gradient flow with a mobility or diffusivity
$b(\varphi)$ depending of the fluid component as
\begin{align}
  \partial_t \varphi &= \mbox{div}(b(\varphi) \nabla \mu),
  \label{eq:CH0_1}\\
  -\epsilon\Delta \varphi + \epsilon^{-1}W^\prime(\varphi) &=
  \mu.\label{eq:CH0_2}
\end{align}
We supplement \eqref{eq:CH0_1}--\eqref{eq:CH0_2} with appropriate initial data
$\varphi_0 \in \mathcal{K}$ and boundary data $\nu_\Omega\cdot \nabla\varphi =
\nu_\Omega \cdot \nabla \mu = 0$ on $\partial \Omega$.
 This system is introduced in
\cite{CahnHilliard} and is called Cahn--Hilliard system.

Specific choices of $b(\varphi)$ and $W(\varphi)$ give different analytical
difficulties. Here we restrict to the case of non
degenerate mobility $b(\varphi)\geq\theta$, for some $\theta>0$, and for
simplicity to constant mobility $b(\varphi)\equiv 1$.

For the free energy we choose  the double-obstacle free energy introduced in
\cite{OonoPuri_DoubleObstacelPotential} and analytically investigated in
\cite{BloweyElliott_I}. 
It has the form 
\begin{align*}
  W(\varphi) := 
\begin{cases}
\frac{1}{2}(1-\varphi^2) & \mbox{if } |\varphi|\leq 1,\\
+\infty & \mbox{else}.
\end{cases}
\end{align*}
Since this $W$ is not smooth, \eqref{eq:CH0_2} here has the form of a
variational inequality, see Definition \ref{def:CH}, where we state the precise formulation
of the Cahn--Hilliard system with constant mobility and double-obstacle free
energy.

%Cahn-Hilliard in continous setting
\begin{definition} 
\label{def:CH}
Let $\Omega \subset \mathbb{R}^n$, $n \in \{2,3\}$, denote a given bounded
domain $I = (0,T]$, $T>0$ denote a given time interval. 
 $\varphi_0 \in \mathcal K$  is a given
 initial phase field.

Then the Cahn--Hilliard system with constant mobility and double-obstacle free
energy consists in finding a phase field $\varphi$ and a chemical potential
$\mu$ such that
\begin{equation}
   \label{eq:CH}
\begin{aligned}
  \varphi \in H^1(0,T,(H^1(\Omega))^\star) \cap L^{\infty}(0,T,H^1(\Omega)),\\
  \mu \in L^2(0,T,H^1(\Omega)),\\
  \varphi(t) \in \mathcal K\quad \forall t \in (0,T),\\
  (\varphi_t,v) + (\nabla \mu,\nabla v) = 0 \quad \forall v \in H^1(\Omega),\\
  \epsilon(\nabla \varphi,\nabla v - \nabla \varphi) -
  \epsilon^{-1}(\varphi,v-\varphi) \geq (\mu,v-\varphi) \quad \forall v \in
  \mathcal K,\\
  \nu_\Omega\cdot\nabla\varphi = 0 \quad \mbox{ on } \partial \Omega,\\
  \nu_\Omega\cdot\nabla\mu = 0 \quad \mbox{ on } \partial \Omega,\\
  \varphi(0) = \varphi_0.
\end{aligned}
\end{equation}
\end{definition}

This system is first investigated in \cite{BloweyElliott_I}. We state the
existence and regularity results here.

\begin{theorem}[{\cite{BloweyElliott_I}}]
There exists a unique solution $\varphi,\mu$ to \eqref{eq:CH} fulfilling
$\varphi \in L^2(0,T; H^2(\Omega))$, $\partial_{\nu_\Omega} \varphi = 0$ on
$\partial \Omega$ for a.e. $t$. It further holds $\forall t>0$
$\min(\sqrt{t},1)\|\varphi(t)\|_{H^2(\Omega)} \leq C(\varphi_0)$, and 
$\min(\sqrt{t},1)\|\mu(t)\|_{H^1(\Omega)}\leq C(\varphi_0)$.
\end{theorem}

In this work we deal with the time discrete variant of \eqref{eq:CH}.
For this let $0=t_0<t_1<\ldots<t_{k-1}<t_k<\ldots < t_M = T$ denote a
decomposition of $I=(0,T]$ with time step size $\tau^k = t_{k}-t_{k-1}$.

Then the time discrete variant of \eqref{eq:CH} in weak form consist of finding
$(\varphi^\star,\mu^\star) \in H^1(\Omega)\times
H^1(\Omega)$ fulfilling
\begin{align}
  (\varphi^\star,v) + \tau^k(\nabla \mu^\star,\nabla v) &= (\varphi^{k-1},v)
  &&\forall v \in H^1(\Omega),\label{eq:CH_td1}\\
  \epsilon(\nabla \varphi^\star,\nabla v-\nabla \varphi^\star) -
  \epsilon^{-1}(\varphi^{k-1},v-\varphi^\star) &\geq (\mu^\star,v-\varphi^\star)
  &&\forall v \in \mathcal K.\label{eq:CH_td2}
\end{align}
We further introduce a parameter $\hp\gg 0$ and the penalising function
\begin{align*}
  \lambda(\varphi) = \lambda_+(\varphi) + \lambda_-(\varphi) 
  := \max(0,\varphi-1) + \min(0,\varphi+1),
\end{align*}
and define the outer approximation, \cite{Glowinski1981},  of 
\eqref{eq:CH_td1}--\eqref{eq:CH_td2} as follows:\\
Find $(\varphi^\hp,\mu^\hp) \in H^1(\Omega) \times H^1(\Omega)$ fulfilling
\begin{align}
  (\varphi^\hp,v) + \tau^k(\nabla \mu^\hp,\nabla v) &= (\varphi^{k-1},v) && \forall
  v \in H^1(\Omega), \label{eq:CH_s1}\\
  \epsilon(\nabla \varphi^\hp,\nabla v) 
  + \epsilon^{-1}(\hp\lambda(\varphi^\hp),v) 
  - \epsilon^{-1}(\varphi^{k-1},v) &=
  (\mu^\hp,v) && \forall v\in H^1(\Omega).\label{eq:CH_s2}
\end{align}

Then the following result holds.

\begin{theorem}[{\cite[Thm. 3.2, Thm. 4.1,Thm. 4.2, Lem. 4.3, Thm.
4.4]{HintermuellerHinzeTber}}] 
\label{thm:allTber}
There exist  unique solutions  $\varphi^\star,\mu^\star$ to
\eqref{eq:CH_td1}--\eqref{eq:CH_td2} and $\varphi^\hp,\mu^\hp$ to 
\eqref{eq:CH_s1}--\eqref{eq:CH_s2}. Moreover 
\begin{align*}
(\varphi^\hp,\mu^\hp) \to (\varphi^\star,\mu^\star)\mbox{ in } H^1(\Omega),  
\end{align*}
and
 there exists a constant $C>0$ independent of $\hp$ such that 
 \begin{align*}
     \|\varphi_\hp\|_{H^1(\Omega)} 
  + \hp\|\lambda(\varphi_\hp)\|_{L^2(\Omega)}
  +\|\mu_\hp\|_{H^1(\Omega)} \leq C.
 \end{align*}

\end{theorem}

% \begin{theorem}[{\cite[Thm. 4.1]{HintermuellerHinzeTber}}]
% The problem \ref{eq:Ps} has a unique solution $(\varphi_\hp,p_\hp)$.
% Moreover there exists a unique Lagrange multiplier $\mu_s\in H^1(\Omega)$ such
% that $p_\hp = \mu_\hp - (\mu_\hp,1)$ and $(\varphi_\hp,\mu_\hp)$ is a
% solution to \eqref{eq:CH_s1}--\eqref{eq:CH_s2}. Conversely, if
% $(\varphi_\hp,\mu_\hp)$ is a solution to \eqref{eq:CH_s1}--\eqref{eq:CH_s2},
% then $(\varphi_\hp,p_\hp)$ with $p_\hp = \mu_\hp - (\mu_\hp,1)$ is the unique
% solution to \ref{eq:Ps}.
% \end{theorem}
% The uniqueness of the solution to \eqref{eq:CH_s1}--\eqref{eq:CH_s2}
% follows from standard techniques using convexity of $\lambda(\varphi)$ and
% concavity of $-\varphi^{k-1}$.

\begin{remark}
We introduce a free energy density
$W^{\hp}(\varphi) = \frac{1}{2}(1-\varphi^2) + \frac{\hp}{2}\lambda(\varphi)^2$.
Then \eqref{eq:CH_s2} can also be obtained as time discretization of
\eqref{eq:CH0_1}--\eqref{eq:CH0_2} with this free energy.
However, we note that $W^\hp$ is not a valid free energy in the sense that its
minima are located at $\theta = \pm\left(1+\frac{\hp}{\hp-1}\right)$ and attain
negative values.
After shifting $W^\hp$ to have non negative values and scaling its argument by
$\theta$, this $W^\hp$ might be regarded as a new type of free energy.

 From the point of treating  variational inequalities,
 \eqref{eq:CH_s2} is an outer approximation \cite{Glowinski1981} of
 \eqref{eq:CH_td2}, while the logarithmic free energy proposed in
 \cite{CahnHilliard} can be regarded as an inner approximation.
\end{remark}

% Last we have the convergence of the solution to the relaxed system
% \eqref{eq:CH_s1}--\eqref{eq:CH_s2} to the original system
% \eqref{eq:CH_td1}--\eqref{eq:CH_td2}.
% 
% \begin{theorem}[{\cite[Thm. 4.2, Lem. 4.3, Thm. 4.4]{HintermuellerHinzeTber}}]
% \label{thm:phimu_converge_bounded}
% Let $(\varphi_\hp, \mu_\hp)_{\hp>0}$ be a sequence of solutions to 
% \eqref{eq:CH_td1}--\eqref{eq:CH_td2} as $\hp\to +\infty$. Then there exists a
% subsequence still denoted by  $(\varphi_\hp, \mu_\hp)_{\hp>0}$ such that
% \begin{align*}
%   (\varphi_\hp,\mu_\hp) \to (\varphi^\star,\mu^\star)\mbox{ in } H^1(\Omega).
% \end{align*} 
% There further exist a constant $C>0$ independent of $\hp$ such that for $\hp\to
% \infty$
% \begin{align*}
%   \|\varphi_\hp\|_{H^1(\Omega)} 
%   + \hp\|\lambda(\varphi_\hp)\|_{L^2(\Omega)}
%   +\|\mu_\hp\|_{H^1(\Omega)} \leq C.
% \end{align*}
% \end{theorem}

\section{A first $L^\infty$ bound on the constraint violation}
\label{sec:Linfty1}
In this section we follow the approach in
\cite[Sec. 2]{HintermuellerSchielaWollner_length_PD_path_MY_path_following}.
Using the uniform boundedness of $\varphi^\hp$ and $\mu^\hp$ from Theorem
\ref{thm:allTber} we first derive $L^1(\Omega)$ bounds for the
constraints violation $\lambda(\varphi^\hp)$ that we later use to derive a first
bound on $\|\lambda(\varphi^\hp)\|_{L^\infty(\Omega)}$.

Note that from regularity theory for the Laplace operator with Neumann boundary
data, see e.g. \cite{Taylor_PDE}, we have $\varphi^\hp\in H^2(\Omega)$ 
and from \eqref{eq:CH_s2} and Theorem \ref{thm:allTber} it follows that
$\|\varphi^\hp\|_{H^2(\Omega)}$ is uniformly bounded in $\hp$.
From Sobolev embedding theory we further have 
$H^2(\Omega)\hookrightarrow C^{0,\beta}(\overline \Omega)$ 
for $\beta < 2-\frac{n}{2}$, and thus indeed
$\lambda(\varphi^\hp) \in L^\infty(\Omega)$. 
From Theorem \ref{thm:allTber} we further have
$\lambda(\varphi^\hp) \to 0$ for $\hp\to \infty$, 
compare also the discussion in
\cite[Rem. 6]{GarckeHinzeKahle_CHNS_AGG_linearStableTimeDisc}.

\begin{theorem}
\label{thm:L1bound}
There exists $C>0$ independent of $\hp$, such that for $\hp \to \infty$  it
holds
\begin{align*}
  \|\lambda(\varphi^\hp)\|_{L^1(\Omega)} \leq C\hp^{-1}.
\end{align*}
\end{theorem}
\begin{proof}
We test \eqref{eq:CH_s1} with $v \equiv \mu^\hp$, \eqref{eq:CH_s2} with
$v\equiv \varphi^\hp$ and add the two equations yielding
\begin{align*}
\epsilon^{-1}\hp(\lambda(\varphi^\hp),\varphi^\hp) &=
-\tau^k\|\nabla \mu^\hp\|^2 
+ (\varphi^{k-1},\mu^\hp) 
-\epsilon \|\nabla \varphi^\hp\|^2
 + \epsilon^{-1}(\varphi^{k-1},\varphi^\hp) \\
 &\leq C 
\end{align*}
with a $C>0$ independent of $\hp$ by the uniform boundedness of
$\varphi^\hp$ and $\mu^\hp$ in $H^1(\Omega)$.

By noting $\varphi^\hp>1 \Leftrightarrow \lambda_+(\varphi^\hp) > 0$
and $\varphi^\hp<-1 \Leftrightarrow \lambda_-(\varphi^\hp) < 0$, 
we further
have
\begin{align*}
  \int_\Omega \lambda_+(\varphi^\hp) \varphi^\hp \dx
  \geq \int_\Omega |\lambda_+(\varphi^\hp)|\dx,\quad
  \int_\Omega \lambda_-(\varphi^\hp) \varphi^\hp \dx
  \geq \int_\Omega |\lambda_-(\varphi^\hp)|\dx,  
\end{align*}
and thus 
\begin{align*}
  \epsilon^{-1}\hp \|\lambda(\varphi^\hp)\|_{L^1(\Omega)} \leq
  \epsilon^{-1}\hp(\lambda(\varphi^\hp),\varphi^\hp)\leq C
\end{align*}
which completes the proof.
\end{proof}

We next derive a bound on $\|\lambda(\varphi^\hp)\|_{L^\infty}$ in terms of
$\|\lambda(\varphi^\hp)\|_{L^1(\Omega)}$. 
Here we follow \cite[Thm. 3.7]{Hintermueller_Keil_Wegner__OPT_CHNS_density},
where \cite[Thm. 2.4]{HintermuellerSchielaWollner_length_PD_path_MY_path_following}
is adapted.

\begin{theorem}\label{thm:Linfty_1}
It holds 
\begin{align*}
  \|\lambda(\varphi^\hp)\|_{L^\infty(\Omega)} 
  \leq
  C(\Omega,n,\beta)\|\lambda(\varphi^\hp)\|_{L^1(\Omega)}^\frac{\beta}{\beta+n}
\end{align*}
\end{theorem}
\begin{proof}
From $\varphi^\hp \in H^2(\Omega)$ we have the H\"older continuity of
$\varphi^\hp \in C^{0,\beta}(\bar{\Omega})$ for $\beta<2-\frac{n}{2}$, and thus
there exists $C_\beta>0$ independent of $\hp$ with
$\|\varphi^\hp\|_{C^{0,\beta}(\Omega)} \leq C_\beta$.

For fixed $\hp$ 
we set $G^+ = \{ x\in \Omega \,|\, \varphi^\hp(x)\geq1\}$ and define
$x_{\max} \in G^+$ to satisfy
\begin{align*}
  \varphi^\hp(x_{\max})-1 = \|\varphi^\hp - 1\|_{L^\infty(G^+)}.
\end{align*}
We set $G^- = \{ x \in \Omega \,|\, \varphi^\hp(x) \leq -1\}$ and define 
$x_{\min} \in G^-$ to satisfy
\begin{align*}
  -(\varphi^\hp(x_{\min})+1) = \|\varphi^\hp + 1\|_{L^\infty(G^-)}.
\end{align*}
Now either $\varphi^\hp(x_{\max})-1 \equiv
\|\lambda(\varphi^\hp)\|_{L^\infty(\Omega)}$ 
or 
$-(\varphi^\hp(x_{\min})+1) \equiv
\|\lambda(\varphi^\hp)\|_{L^\infty(\Omega)}$ holds.
W.l.o.g. we assume 
$\varphi^\hp(x_{\max})-1 \equiv
\|\lambda(\varphi^\hp)\|_{L^\infty(\Omega)}$.

From the definition of H\"older continuity we have  
\begin{align*}
  \varphi^\hp(x)-1 \geq \varphi^\hp(x_{\max}) -1
  -\|\varphi^\hp\|_{C^{0,\beta}(\Omega)}|x_{\max}-x|^\beta.
\end{align*}
Thus, for $|x_{\max}-x| \leq 
\left(\frac{\varphi^\hp(x_{\max})-1}{2C_\beta} \right)^{1/\beta}$
\begin{align*}
  \varphi^\hp(x)-1 \geq \frac{1}{2}(\varphi^\hp(x_{\max})-1) > 0
\end{align*}
holds.

The domain $\Omega$ satisfies the cone condition. Thus there exists a cone
$K_r(x_{\max}) := K(x_{\max}) \cap B(x_{\max},r)$ of radius $r$ and with vertex
$x_{\max}$ such that $K_r(x_{\max}) \subset \Omega$.
Hence the cone $K_R(x_{\max})$ with
$R:=\min\left(r,\left(\frac{\|\lambda(\varphi^\hp)\|_{L^\infty(\Omega)}}{2C_\beta}
\right)^{1/\beta}\right)$ is contained in $G$.

From this we conclude
\begin{align*}
  \|\lambda(\varphi^\hp)\|_{L^1(\Omega)}
   &\geq   \int_{K_R(x_{\max})} \varphi^\hp-1\dx\\
   &\geq \int_ {K_R(x_{\max})} \frac{1}{2}(\varphi^\hp(x_{\max})-1)\dx\\
   &= \frac{|K_R(x_{\max})|}{2}\|\lambda(\varphi^\hp)\|_{L^\infty(\Omega)}\\
   &\geq C(\Omega,n) 
   \left(\frac{\|\lambda(\varphi^\hp)\|_{L^\infty(\Omega)}}{2C_\beta} \right)^{n/\beta}
   \|\lambda(\varphi^\hp)\|_{L^\infty(\Omega)}\\
   &=
   C(\Omega,n,\beta)\|\lambda(\varphi^\hp)\|^{\frac{n+\beta}{\beta}}_{L^\infty(\Omega)}
\end{align*}

\end{proof}

\begin{corollary}
From Theorem \ref{thm:Linfty_1} we have the following $L^\infty$ bounds for the
violation of the constraint $|\varphi| \leq 1$.
\begin{align*}
  \|\lambda(\varphi^\hp)\|_{L^\infty(\Omega)} \leq 
  \begin{cases}
  C\hp^{-1/3+\gamma} & \mbox{ if } n = 2,\\
  C\hp^{-1/7+\gamma} & \mbox{ if } n = 3,
  \end{cases}
\end{align*}
 where $0<\gamma\ll 1$.
\end{corollary}
\begin{proof}
It holds 
\begin{align*}
  \|\lambda(\varphi^\hp)\|_{L^\infty(\Omega)} 
  \leq C \|\lambda(\varphi^\hp)\|_{L^1(\Omega)}^{\frac{\beta}{n+\beta}}
  \leq C \hp^{-\frac{\beta}{n+\beta}}.
\end{align*}
For $n=2$ we have $\beta<1$ and thus the first claim.\\
For $n=3$ we have $\beta<\frac{1}{2}$ and thus the second claim.
\end{proof}

\section{A structural assumption to improve the $L^\infty$ bound}
\label{sec:Linfty2}
In the previous section we derived a bound of the violation
$\|\lambda(\varphi^\hp)\|_{L^\infty(\Omega)}$ based on its $\|\lambda(\varphi^\hp)\|_{L^1(\Omega)}$ norm.
We now state an assumption on the order of $\|\lambda(\varphi^\hp)\|_{L^1(\Omega)}$ to obtain an
improved $L^\infty(\Omega)$ bound.

\begin{assumption}
[{\cite[Ass. 2.13]{HintermuellerSchielaWollner_length_PD_path_MY_path_following}}]
\label{ass:structure}
For $x \in \Omega$ we assume that there exists $0<K<\infty$, independent of
$x,R$, and $\hp$ such that
\begin{align*}
  \int_{B_R(x)} \hp  |\lambda(\varphi^\hp)|\dx \leq KR^n
\end{align*}
holds.
\end{assumption}
This assumption states, that $G = \{ x\in \Omega \,|\, \lambda(\varphi^\hp(x))
\neq 0 \}$ is a subset of $\Omega$ that has a  measure comparable to the
measure of $\Omega$ and is not a lower dimensional manifold.
 This is a reasonable assumption since in our application $G$ describes the
 bulk phases which fill most parts of the domain.

\begin{theorem}\label{thm:Linfty_2}
Let Assumption \ref{ass:structure} hold.
Then it holds
\begin{align*}
  \|\lambda(\varphi^\hp)\|_{L^\infty(\Omega)} 
  \leq C(\Omega,n,\beta,K)\hp^{-1}.
\end{align*}
\end{theorem}
\begin{proof}
Using Assumption \ref{ass:structure} in the proof of Theorem \ref{thm:Linfty_1}
we obtain
\begin{align*}
  K\hp^{-1}R^n &\geq \|\lambda(\varphi^\hp)\|_{L^1(\Omega)}\\
&\geq
C(\Omega,n,\beta)\|\lambda(\varphi^\hp)\|_{L^\infty(\Omega)}^{\frac{\beta+n}{\beta}}
\end{align*} 
Restating the definition of
 $R:=\min\left(r,\left(\frac{\|\lambda(\varphi^\hp)\|_{L^\infty(\Omega)}}{2C_\beta}
\right)^{1/\beta}\right)$ and 
solving for
$\|\lambda(\varphi^\hp)\|_{L^\infty(\Omega)}$ we obtain
\begin{align}
  \|\lambda(\varphi^\hp)\|_{L^\infty(\Omega)} \leq C(\Omega,n,\beta,K)\hp^{-1}.
  \label{eq:LinfImproved}
\end{align}

\end{proof}

Note that this bound is significantly better than the one shown in Theorem
\ref{thm:Linfty_1} and especially is independent of the  dimension $n$ of
the domain $\Omega$ and of the H\"older exponent $\beta$ of $\varphi^\hp$.

\begin{remark} \label{rm:higherPowersLambda}
In \eqref{eq:CH_s2} we use the violation $\lambda(\varphi^\hp)$ as penalisation
for the outer approximation of the variational inequality \eqref{eq:CH_td2}. 
The function $\lambda$ is Lipschitz continuous, i.e.
$\lambda \in C^{0,1}(\mathbb{R})$.
 We 
introduce smoother penalisations by using higher powers of $\lambda$, i.e. for
$k\geq 2$ we define 
\begin{align*}
  \lambda_k(\varphi^\hp) :=  
  \lambda(\varphi^\hp)|\lambda(\varphi^\hp)|^{k-2}.
\end{align*}
Note that $\lambda_2(\varphi^\hp) \equiv \lambda(\varphi^\hp)$ and that it holds
$\lambda_k \in C^{k-2,1}(B)$ for each bounded interval $B\subset \mathbb{R}$.
Further we have at least $\varphi^\hp \in H^1(\Omega) \hookrightarrow
L^q(\Omega)$ for $q<6$ if $n\equiv 3$ and $q<\infty$ for $n\equiv 2$. Thus these
higher powers of $\lambda$ are still integrable.

Using $\lambda_k$ instead of $\lambda$ in \eqref{eq:CH_s2} leads to the same
results, but  with $\lambda$ substituted by $\lambda_k$. 
Especially we obtain
\begin{align}
  \|\lambda_k(\varphi^\hp)\|_{L^\infty(\Omega)}
   = \|\lambda(\varphi^\hp)^{k-1}\|_{L^\infty(\Omega)}
   \leq C\hp^{-1}
\end{align}
and thus 
\begin{align*}
  \|\lambda(\varphi^\hp)\|_{L^\infty(\Omega)} \leq C\hp^{-\frac{1}{k-1}}
\end{align*}
Such penalisations with smoother functions might be of interest for optimal
control problems, \cite{HintermuellerKopacka,Hintermueller_Keil_Wegner__OPT_CHNS_density}, where sufficient
differentiability properties are required,
 or in the
context of model order reduction using proper orthogonal decomposition,
\cite{Volkwein_POD}. 
\end{remark}

\subsection{The fully discrete case}
\label{ssec:fullyDisc}
We assume, that $\Omega$ is polygonally bounded and thus can be represented
exactly by a finite number of triangles if $n=2$, resp. tetrahedrons if $n=3$.
We introduce a subdivision $\mathcal T$ of $\Omega$ using triangles, resp.
tetrahedrons, and define the space of piecewise linear and globally continuous
finite elements as
\begin{align*}
  V_h = \{v \in C(\bar{\Omega}) \,|\, 
  v|_T \in P^1(T)\,\forall T\in \mathcal{T} \} 
  = \mbox{span}\{\phi_i\,|\,  i=1,\ldots,N\},
\end{align*}
where $P^1(T)$ denotes the space of all polynomials defined on $T$ up to order
1.
Note that functions from $V_h$ are H\"older continuous with exponent $\beta<1$.

The discrete counterparts $\varphi^\hp_h, \mu^\hp_h$ of
$\varphi^\hp,\mu^\hp$ then fulfil the equations
\begin{align*}
    ( \varphi^\hp_h,v) + \tau^k(\nabla \mu^\hp_h,\nabla v) &=
    (\varphi^{k-1},v) && \forall v \in V_h,\\
  \epsilon(\nabla (\varphi^\hp_h,\nabla v) +
  \epsilon^{-1}\Lambda(\varphi^\hp_h,v)  -\epsilon^{-1} (\varphi^{k-1},v)
  &= (\mu^\hp_h,v) && \forall v\in V_h,
\end{align*}
where we investigate three choices for the numerical treatment of the
penalisation term $\Lambda_k(\varphi^\hp_h,v)$.
Using linear elements the term
$\Lambda(\varphi^\hp_h,v) = \Lambda_E(\varphi^\hp_h,v) :=
(\hp\lambda(\varphi^\hp_h),v)$ can be evaluated exactly. Another possibility is
to use the Lagrangian interpolation $I$ for $\lambda(\varphi^\hp_h)$, i.e.
$\Lambda(\varphi^\hp_h,v) = \Lambda_I(\varphi^\hp_h,v) := (\hp
I(\lambda(\varphi^\hp_h)),v)$.
As last variant we investigate an evaluation using lumping, i.e.
$\Lambda(\varphi^\hp_h,v) = \Lambda_L(\varphi^\hp_h,v) := (\hp
\lambda(\varphi^\hp_h),v)^h$, where $(v,w)^h$ is defined as
\begin{align*}
  (v,w)^h := \sum_{i =1}^N (1,\phi_i)v(x_i)w(x_i), 
\end{align*}
and $x_i$ stands for  the vertices of $\mathcal T$.

Note that $|I(\lambda(\varphi^\hp_h))| \geq |\lambda(\varphi^\hp_h)| $ and 
$\|I(\lambda(\varphi^\hp_h))\|_{L^\infty(\Omega)} \equiv \|\lambda(\varphi^\hp_h)\|_{L^\infty(\Omega)}$ due to the linear finite elements
that we use.

Using the different variants of $\Lambda$ the results from Section
\ref{sec:Linfty1} and Section \ref{sec:Linfty2} stay valid for the fully
discrete case, with small changes that we comment on next.

For $\Lambda_E$ no changes are required.
For $\Lambda_I$ and $\Lambda_L$ we obtain  from Theorem \ref{thm:L1bound}
\begin{align*}
\epsilon^{-1}\hp\|\lambda(\varphi^\hp_h)\|_{L^1(\Omega)} \leq 
  \epsilon^{-1}\hp\|I(\lambda(\varphi^\hp_h))\|_{L^1(\Omega)} \leq C.
\end{align*}
The according $L^\infty(\Omega)$ bounds for $\lambda(\varphi^\hp_h)$ in
\eqref{eq:LinfImproved} then follow immediately.

\section{Numerical validation}
\label{sec:numerics}
In this section we perform numerical tests to validate the bounds that we found
in Theorem \ref{thm:Linfty_2}.
 
We perform simulations in two and three space dimensions, using $\Omega =
(0,1)^n$, i.e. the unit square for $n=2$, and the unit cube for $n=3$.

We perform one time step of \eqref{eq:CH_s1}--\eqref{eq:CH_s2} and measure the
violation $\|\lambda(\varphi^\hp)\|_{L^\infty(\Omega)}$. 
The parameter are given as $\epsilon = 0.01$ and $\tau = 0.01$ if $n=2$ and
$\epsilon = 0.04$ and $\tau=0.01$ if $n=3$.
The initial phase field $\varphi_0$
is defined as
\begin{align*}
  \varphi_0(x) &= 
  \begin{cases}
  1 & \mbox{ if } z \geq  \frac{\pi}{2},\\
  -1 & \mbox{ if } z \leq  -\frac{\pi}{2},\\
  \sin( z )& \mbox{ else, }
  \end{cases}
\end{align*}
where $ z =
  \epsilon^{-1}(\|x-m\|-r) $, and $m = (0.5,0.5)$ for $n=2$ and
  $m=(0.5,0.5,0.5)$ if $n=3$. This defines a sphere with radius $r = 0.25$. Note that the sinus is the principle shape of
the phase field across the interface for the double-obstacle free energy
\cite{ElliotStinnerStylesWelford_AdvecDiffEvolvingInterface}.

We adapt the mesh for  the phase field and the chemical potential by using the
reliable and efficient residual based error estimator proposed in
\cite{HintermuellerHinzeTber}. Here we only use the norm of the jumps of the
normal derivatives of $\varphi_h^\hp$ and $\mu^\hp_h$ across edges for $n=2$,
resp. faces for $n=3$, which contains the main part of the indicator
\cite{CarstensenVerfuerth_EdgeResidualDominate}. For the marking of cells we
follow \cite{HintermuellerHinzeTber}, i.e. we the procedure proposed by 
D\"orfer \cite{Doerfler}.

 We perform the usual 
``solve$\to$estimate$\to$mark$\to$refine''
 cycle
three times for each value of $\hp$ and reuse the final mesh as initial mesh for the next larger value of $\hp$.
In Figure \ref{fig:2d:grid} we show a typically mesh obtained by this adaptive
solving procedure.

\begin{figure}
  \centering 
  \includegraphics[width=0.5\textwidth]{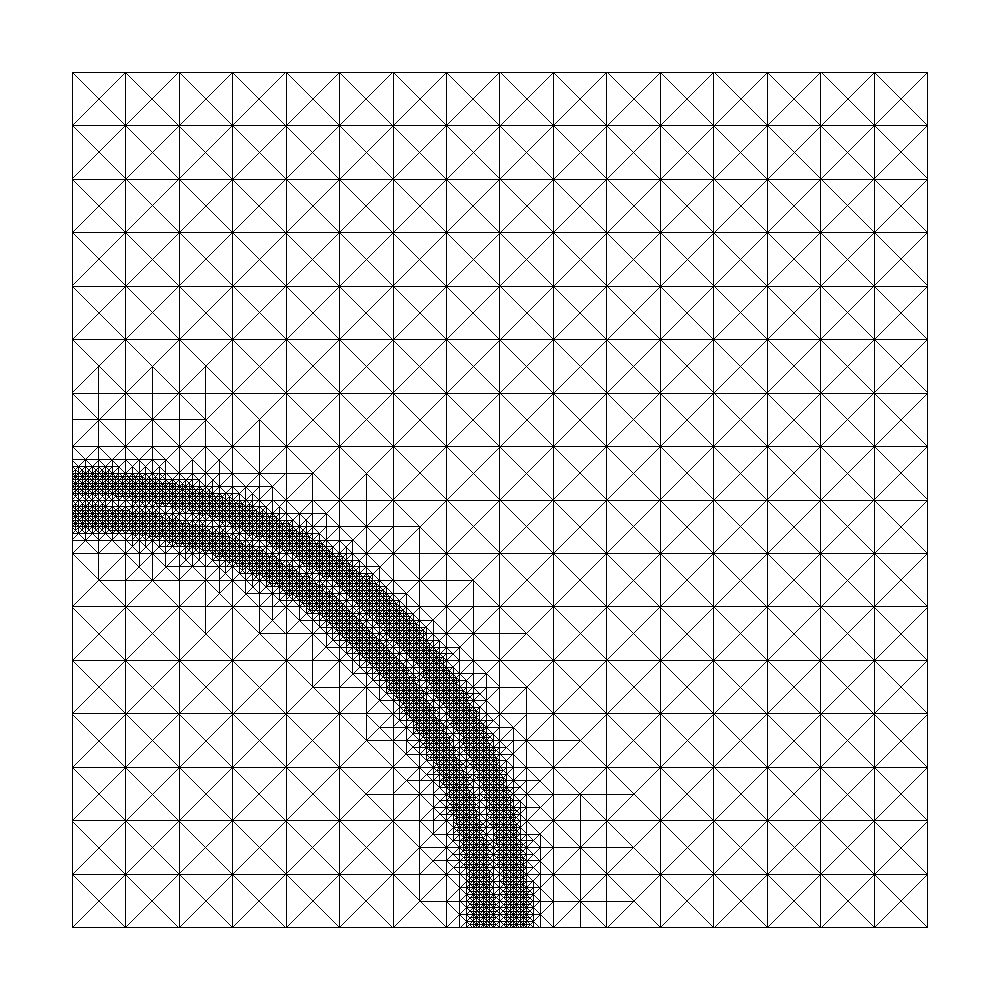}
  \caption{One quarter of the computational domain for the 2d simulation.}
  \label{fig:2d:grid}
\end{figure}

The non linear system \eqref{eq:CH_s1}--\eqref{eq:CH_s2} is solved using
Newton's method and the linear systems are solved directly. The implementation
is done in C++ using the finite element toolbox FEniCS \cite{fenics_book} with the PETSc
\cite{petsc_webpage} linear algebra back end and the MUMPS \cite{mumps} direct
solver.

\subsection*{Experiments in two space dimensions}
In two space dimensions we investigate the three cases of numerical treatment of
the penalisation $\lambda(\varphi^\hp)$ as proposed in Section
\ref{ssec:fullyDisc}, namely the exact integration, the interpolation case and
the lumping case.

In Figure \ref{fig:2d:lambda2} we present numerical results. We show the
violation $\|\lambda(\varphi^\hp)\|_{L^\infty(\Omega)}$ for the different
numerical treatments of terms involving $\lambda$.

\begin{figure}
  \centering
  \includegraphics[width=0.45\textwidth]{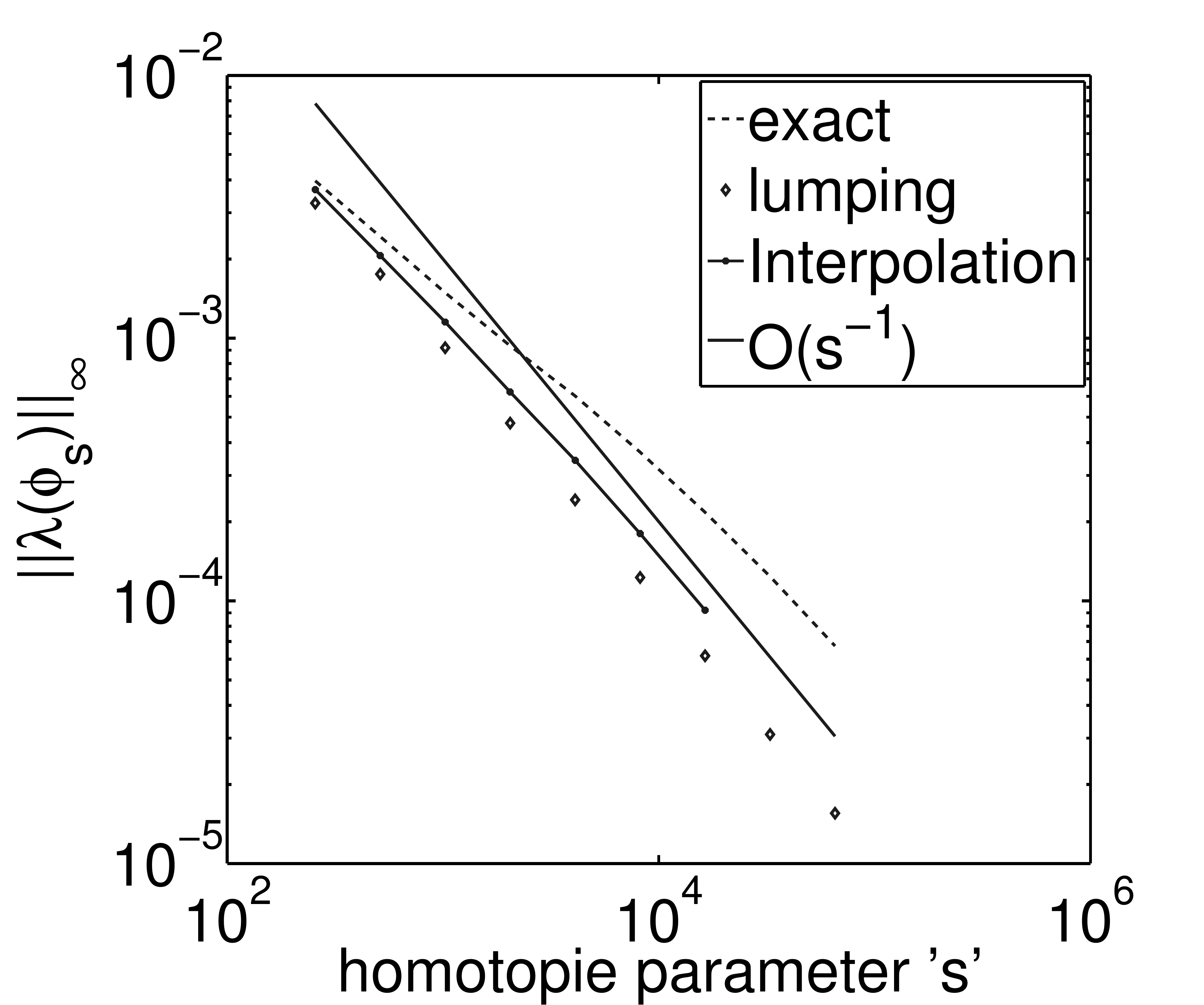}
  \caption{The violation of the bound $|\varphi|\leq 1$ for penalisation with
  $\lambda_2$, and the different numerical
  treatments, i.e. exact evaluation, lumping and interpolation.
  Note that  Newton's method failed in calculating the solution
  for larger values of $\hp$ if interpolation is used.
  }
  \label{fig:2d:lambda2}
\end{figure}

In the case of interpolation we see that  Newton's method is not successful in
finding the solution to  \eqref{eq:CH_s1}--\eqref{eq:CH_s2} for large values of
$\hp$,
while the exact evaluation and also the lumping evaluation converge equally
well. In the case of exact integration we have a slightly higher violation of
the constraints.
However in all three cases the theoretical bound from Theorem
\ref{thm:Linfty_2} holds.

\begin{figure}
    \hfill
  \includegraphics[width=0.45\textwidth]{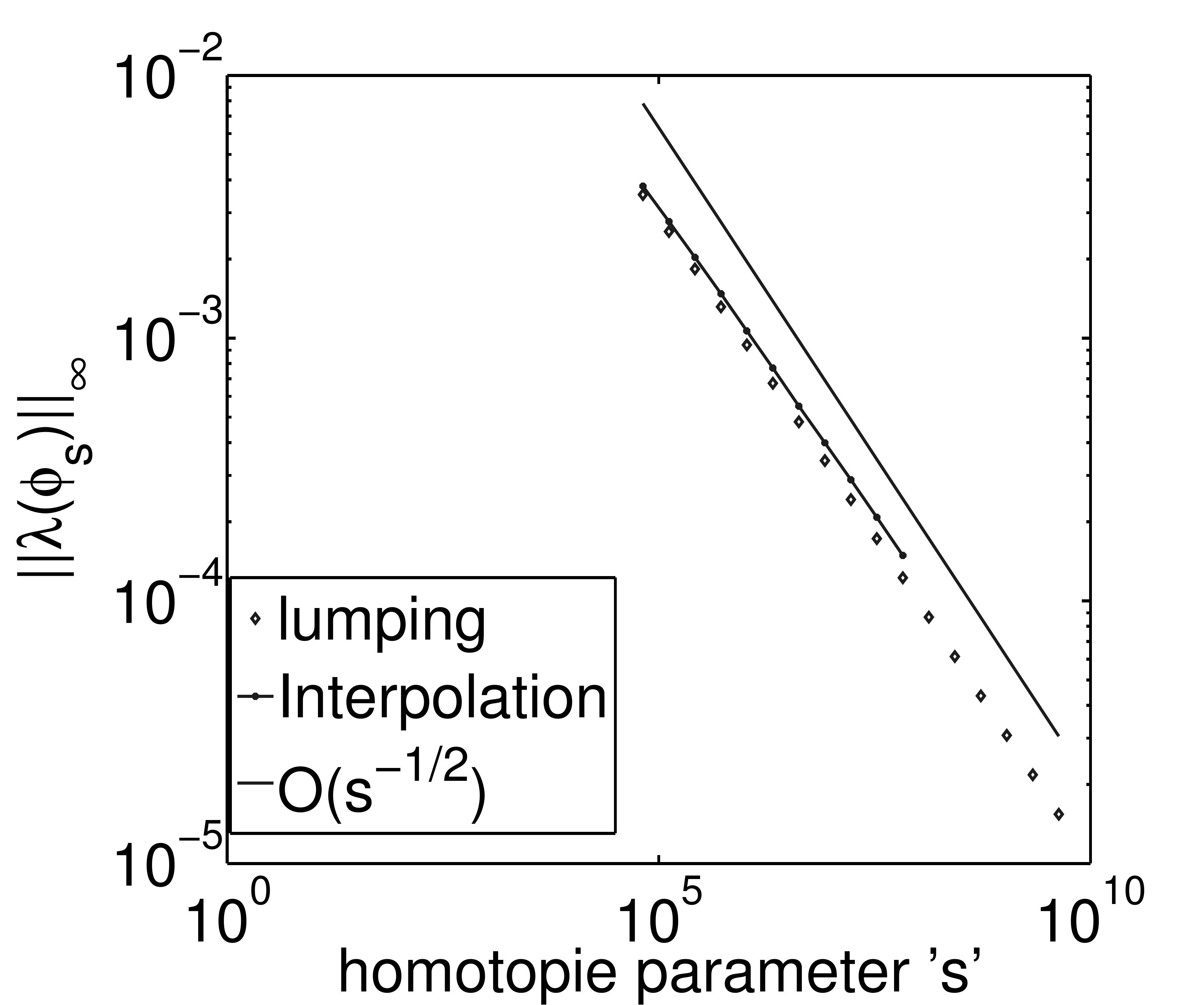}
  \hfill
  \includegraphics[width=0.45\textwidth]{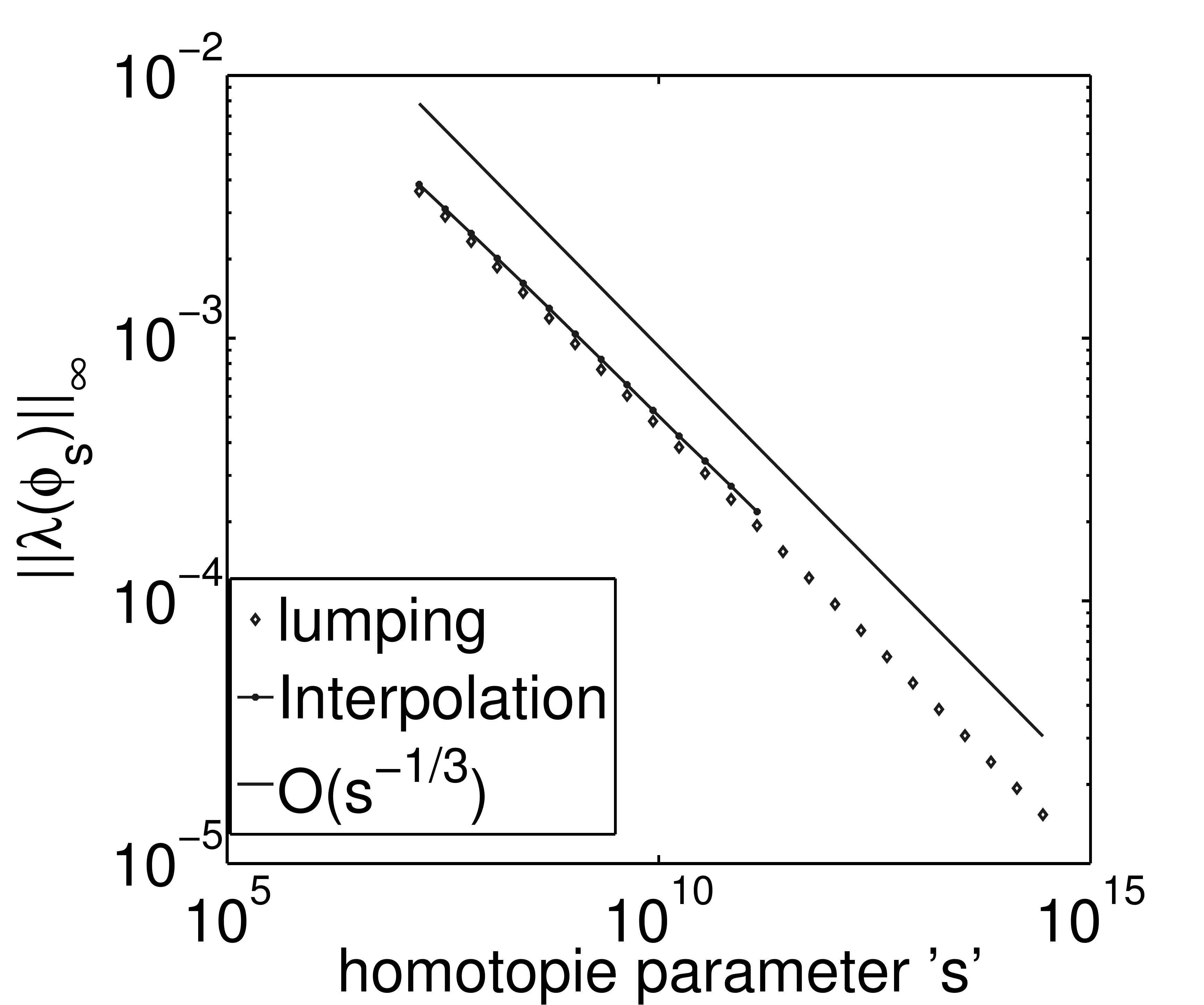}
  \hfill 
  \caption{The violation of the bound $|\varphi|\leq 1$ for penalisation with
  $\lambda_k$, $k=3,4$ (left to right) and the different numerical
  treatments, i.e. lumping and interpolation.
  Note that in both cases, Newton's method fails in calculating the solution
  for larger values of $\hp$ if interpolation is used.
  }
  \label{fig:2d:lambda34}
\end{figure}

As next test we investigate $\|\lambda(\varphi^\hp)\|_{L^\infty(\Omega)}$ when
$\lambda_k$ for $k>2$ is used for penalisation as proposed in Remark
\ref{rm:higherPowersLambda}. Here we do not perform exact integration.

Also in the case of higher powers of $\lambda$ 
 the theoretical bounds are attained. Again we
observe that Newton's method fails in finding the solution for large values of
$\hp$ if interpolation is used, while with lumped evaluation the solution is
always found. Note that to obtain comparable violations for $\lambda_k$ we need
larger values of $\hp$.

\subsection*{Experiments in three space dimensions}
In three dimensions we only investigate lumping and interpolation as treatment
of the term with $\lambda$ and again we use $\lambda_k$ for $k=2,3,4$.

In Figure \ref{fig:3d:lambda234} we show the numerical results. Again we obtain
the expected bounds. Also using interpolation Newton's method fails in finding
the unique solution for larger values of $\hp$.

\begin{figure}
    \hfill 
  \includegraphics[width=0.32\textwidth]{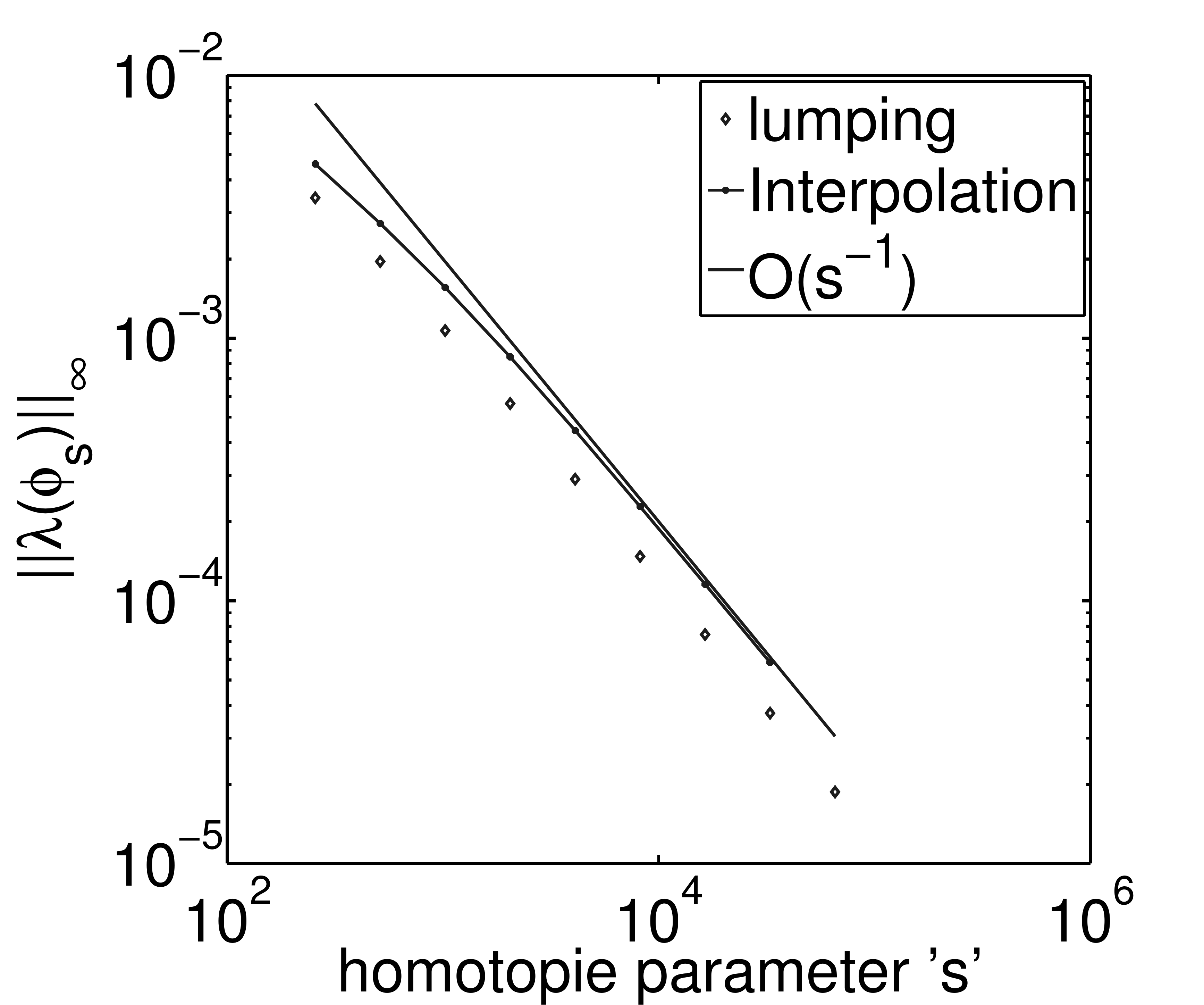}
    \hfill
   \includegraphics[width=0.32\textwidth]{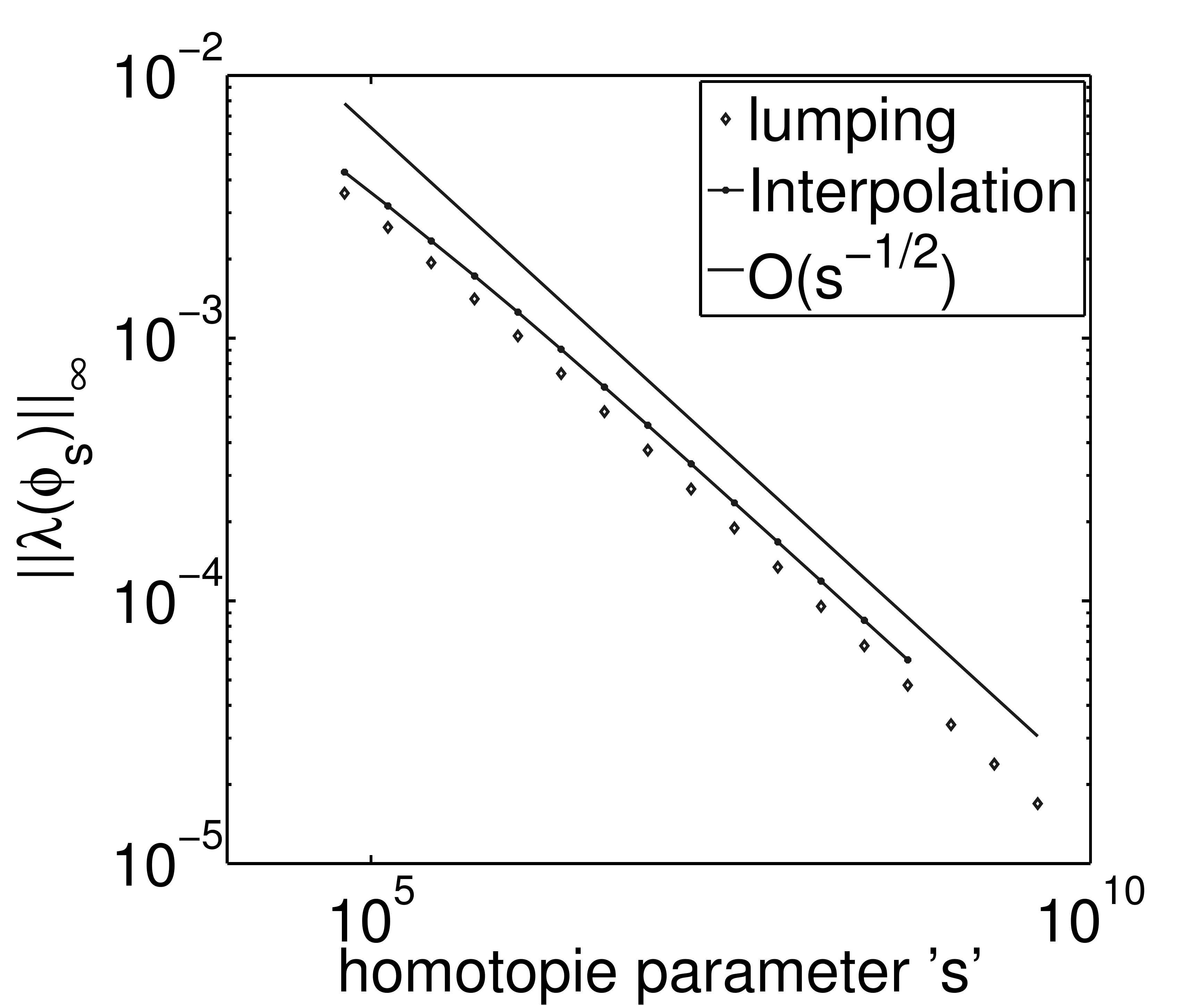}
     \hfill
        \includegraphics[width=0.32\textwidth]{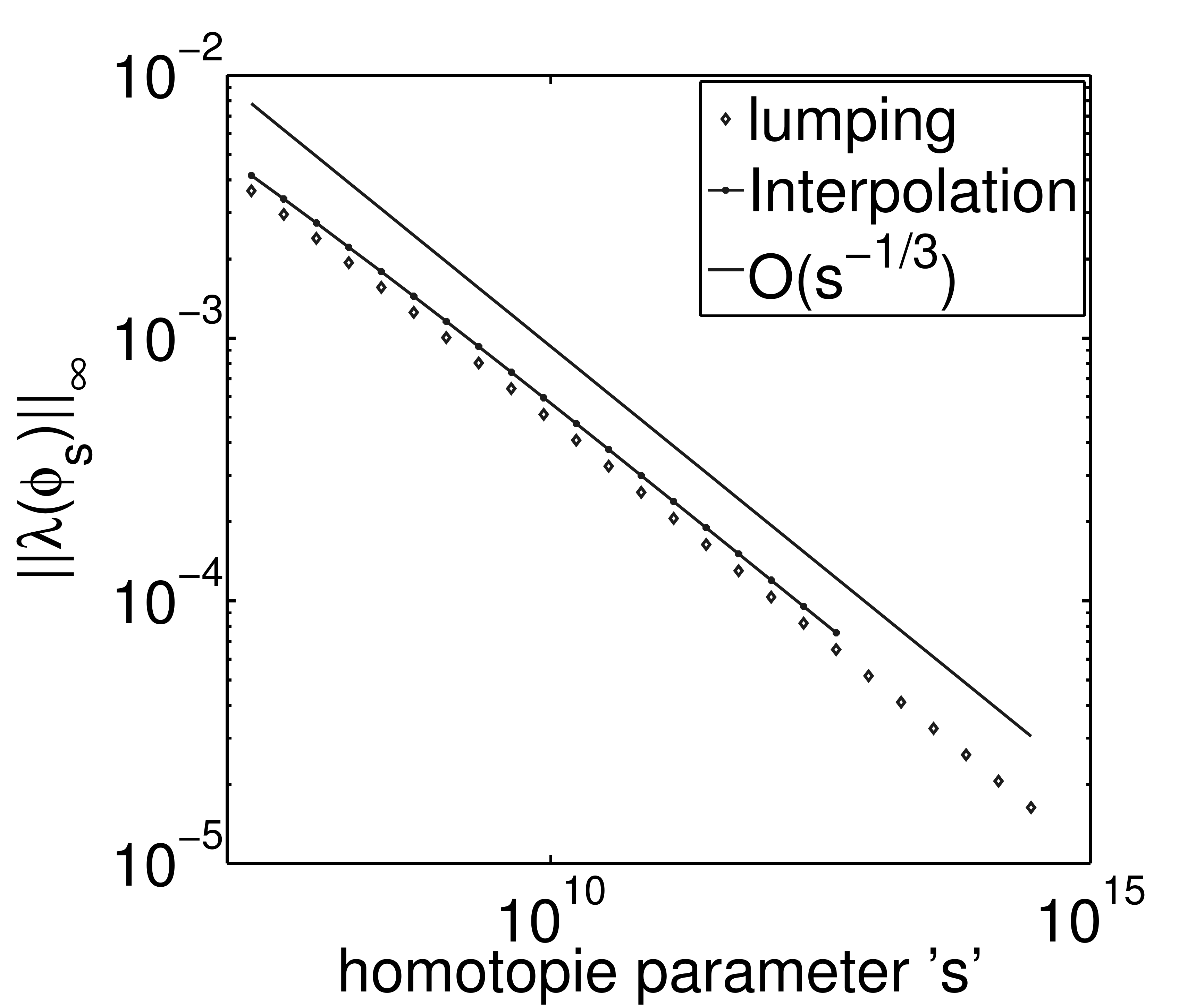}
     \hfill
  \caption{Numerical results for three space dimensions and with $\lambda_k$,
  $k=2,3,4$ (left to right). The proposed convergence rate is obtained and
  again in  the case of interpolation of $\lambda$ Newton's method fails in
  finding the solution for large values of $\hp$.  }
  \label{fig:3d:lambda234}
\end{figure}

\newcommand{\etalchar}[1]{$^{#1}$}


\begin{thebibliography}{BAA{\etalchar{+}}14}

\bibitem[ADKL01]{mumps}
P.R. Amestoy, I.S. Duff, J.~Koster, and J.-Y. L'Excellent.
\newblock A fully asynchronous multifrontal solver using distributed dynamic
  scheduling.
\newblock {\em SIAM Journal of Matrix Analysis and Applications}, 23(1):15--41,
  2001.

\bibitem[AF03]{Adams_SobolevSpaces}
R.~A. Adams and J.~H.~F. Fournier.
\newblock {\em Sobolev Spaces, second edition}, volume 140 of {\em Pure and
  Applied Mathematics}.
\newblock Elsevier, 2003.

\bibitem[AGG12]{AbelsGarckeGruen_CHNSmodell}
H.~Abels, H.~Garcke, and G.~Gr{\"u}n.
\newblock Thermodynamically consistent, frame indifferent diffuse interface
  models for incompressible two-phase flows with different densities.
\newblock {\em Mathematical Models and Methods in Applied Sciences}, 22(3):40,
  March 2012.

\bibitem[BAA{\etalchar{+}}14]{petsc_webpage}
S.~Balay, S.~Abhyankar, M.F. Adams, J.~Brown, P.~Brune, K.~Buschelman,
  L.~Dalcin, V.~Eijkhout, W.D. Gropp, D.~Kaushik, M.G. Knepley, L.C. McInnes,
  K.~Rupp, B.F. Smith, S.~Zampini, and H.~Zhang.
\newblock {PETS}c {W}eb page.
\newblock \url{http://www.mcs.anl.gov/petsc}, 2014.

\bibitem[BE91]{BloweyElliott_I}
J.~F. Blowey and C.~M. Elliott.
\newblock The {Cahn--Hilliard} gradient theory for phase separation with
  non-smooth free energy. {Part I}: Mathematical analysis.
\newblock {\em European Journal of Applied Mathematics}, 2:233--280, 1991.

\bibitem[Boy02]{Boyer_two_phase_different_densities}
F.~Boyer.
\newblock A theoretical and numerical model for the study of incompressible
  mixture flows.
\newblock {\em Computers \& Fluids}, 31(1):41--68, January 2002.

\bibitem[CH58]{CahnHilliard}
J.~W. Cahn and J.~E. Hilliard.
\newblock {Free Energy of a Nonuniform System. I. Interfacial Free Energy}.
\newblock {\em The Journal of Chemical Physics}, 28(2):258--267, 1958.

\bibitem[CV99]{CarstensenVerfuerth_EdgeResidualDominate}
C.~Carstensen and R.~Verf\"urth.
\newblock {Edge Residuals Dominate A Posteriori Error Estimates for Low Order
  Finite Element Methods}.
\newblock {\em SIAM Journal on Numerical Analysis}, 36(5):1571--1587, 1999.

\bibitem[D{\"o}r96]{Doerfler}
W.~D{\"o}rfler.
\newblock A convergent adaptive algorithm for {Poisson's} equation.
\newblock {\em SIAM Journal on Numerical Analysis}, 33(3):1106--1124, 1996.

\bibitem[DSS07]{Ding_Spelt_Shu_diffuse_interface_model_diff_density}
H.~Ding, P.~D.~M. Spelt, and C.~Shu.
\newblock Diffuse interface model for incompressible two-phase flows with large
  density ratios.
\newblock {\em Journal of Computational Physics}, 226(2):2078--2095, October
  2007.

\bibitem[ESSW11]{ElliotStinnerStylesWelford_AdvecDiffEvolvingInterface}
C.M. Elliott, B.~Stinner, V.~Styles, and R.~Welford.
\newblock Numerical computation of advection and diffusion on evolving diffuse
  interfaces.
\newblock {\em IMA Journal of Numerical Analysis}, 31(3):786--812, July 2011.

\bibitem[GHK16]{GarckeHinzeKahle_CHNS_AGG_linearStableTimeDisc}
H.~Garcke, M.~Hinze, and C.~Kahle.
\newblock {A stable and linear time discretization for a thermodynamically
  consistent model for two-phase incompressible flow}.
\newblock {\em Applied Numerical Mathematics}, 99:151--171, January 2016.

\bibitem[GLT81]{Glowinski1981}
R.~Glowinski, J.-L. Lions, and R.~Tr\'emoli\'eres.
\newblock {\em Numerical Analysis of Variational Inequalities}, volume~8 of
  {\em Studies in Mathematics and its Applications}.
\newblock North-Holland Publishing Company, 1981.

\bibitem[HH77]{HohenbergHalperin1977}
P.~C. Hohenberg and B.~I. Halperin.
\newblock Theory of dynamic critical phenomena.
\newblock {\em Reviews of Modern Physics}, 49(3):435--479, 1977.

\bibitem[HHT11]{HintermuellerHinzeTber}
M.~Hinterm\"uller, M.~Hinze, and M.~H. Tber.
\newblock An adaptive finite element {Moreau--Yosida-based} solver for a
  non-smooth {Cahn--Hilliard} problem.
\newblock {\em Optimization Methods and Software}, 25(4-5):777--811, 2011.

\bibitem[HK11]{HintermuellerKopacka}
M.~Hinterm\"uller and I.~Kopacka.
\newblock A smooth penalty approach and a nonlinear multigrid algorithm for
  elliptic { MPECS }.
\newblock {\em Comput. Optim. Appl.}, 50(1):111--145, 2011.

\bibitem[HKW15]{Hintermueller_Keil_Wegner__OPT_CHNS_density}
M.~Hinterm\"uller, T.~Keil, and D.~Wegner.
\newblock {Optimal Control of a Semidiscrete Cahn-Hilliard-Navier-Stokes System
  with Non-Matched Fluid Densities}.
\newblock {\em arXiv: 1506.03591}, 2015.

\bibitem[HSW14]{HintermuellerSchielaWollner_length_PD_path_MY_path_following}
M.~Hinterm\"uller, A.~Schiela, and W.~Wollner.
\newblock {The Length of the Primal-Dual Path in Moreau--Yosida-Based
  Path-Following Methods for State Constrained Optimal Control}.
\newblock {\em SIAM Journal on Optimization}, 24(1):108--126, 2014.

\bibitem[LMW12]{fenics_book}
A.~Logg, K.-A. Mardal, and G.~Wells, editors.
\newblock {\em {Automated Solution of Differential Equations by the Finite
  Element Method - The FEniCS Book}}, volume~84 of {\em Lecture Notes in
  Computational Science and Engineering}.
\newblock Springer, 2012.

\bibitem[LT98]{Lowengrub_CahnHilliard_and_Topology_transitions}
J.~Lowengrub and L.~Truskinovsky.
\newblock Quasi-incompressible {Cahn--Hilliard} fluids and topological
  transitions.
\newblock {\em Proceedings of the royal society A}, 454(1978):2617--2654, 1998.

\bibitem[OP88]{OonoPuri_DoubleObstacelPotential}
Y.~Oono and S.~Puri.
\newblock {Study of phase-separation dynamics by use of cell dynamical systems.
  I. Modeling}.
\newblock {\em Physical Review A}, 38(1):434--463, 1988.

\bibitem[Tay96]{Taylor_PDE}
M.~E. Taylor.
\newblock {\em Partial Differential Equations I: Basic Theory}, volume 115 of
  {\em Applied Mathematical Sciences}.
\newblock Springer, 1996.

\bibitem[Vol11]{Volkwein_POD}
S.~Volkwein.
\newblock Model reduction using proper orthogonal decomposition.
\newblock lecture script, Universit\"at Konstanz, 2011.

\end{thebibliography}
\end{document}